\documentclass{birkjour}
\usepackage{amsfonts, amssymb, amsthm,graphicx,amsmath,amsfonts,color}
\DeclareMathOperator{\grad}{grad}

\theoremstyle{plain}
\newtheorem{teorema}{Theorem}[section]
\newtheorem{proposicion}[teorema]{Proposition}
\newtheorem{lema}[teorema]{Lemma}
\newtheorem{corolario}[teorema]{Corollary}
\newtheorem{ejemplo}[teorema]{Example}
   
\theoremstyle{definition}
\newtheorem{definicion}[teorema]{Definition}
\newtheorem{remark}[teorema]{Remark}

\begin{document}

\title{Null screen quasi-conformal hypersurfaces in semi-Riemannian manifolds and applications}

\author[M. Navarro]{Matias Navarro}
\address{%
Facultad de Matem\'aticas \\
Universidad Aut\'onoma de Yucat\'an - UADY \\
Perif\'erico Norte, Tablaje 13615 \\
M\'erida, Yucat\'an \\
M\'exico}
\email{matias.navarro@correo.uady.mx}

\author[O. Palmas]{Oscar Palmas}
\address{%
Departamento de Matem\'aticas \\
Facultad de Ciencias \\
Universidad Nacional Aut\'onoma de M\'exico - UNAM \\
CP 04510, Ciudad de M\'exico \\
M\'exico}
\email{oscar.palmas@ciencias.unam.mx}

\author[D. Solis]{Didier A. Solis}
\address{%
Facultad de Matem\'aticas \\
Universidad Aut\'onoma de Yucat\'an - UADY \\
Perif\'erico Norte, Tablaje 13615 \\
M\'erida, Yucat\'an \\
M\'exico}
\email{didier.solis@correo.uady.mx}

\thanks{M. Navarro was partially supported by UADY under Project FMAT-2017-0003. O. Palmas was partially supported by UNAM under Project PAPIIT-DGAPA IN113516. D. A. Solis was partially supported by UADY under Projects P/PFCE-2017-31MSU0098J-13 and CEA-SAB-011-2017.}
\subjclass{Primary 53B30; Secondary 53C50}
\keywords{Null hypersurfaces, isoparametric hypersurfaces, Einstein hypersurfaces}

\maketitle

\begin{abstract}
{We introduce a class of null hypersurfaces of a semi-Riemannian manifold, namely, screen quasi-conformal hypersurfaces, whose geometry may be studied through the geometry of its screen distribution. In particular, this notion allows us to extend some results of previous works to the case in which the sectional curvature of the ambient space is different from zero. As applications, we study umbilical, isoparametric and Einstein null hypersurfaces in Lorentzian space forms and provide several classification results.}
\end{abstract}

\section{Introduction}

The existence of null hypersurfaces is one of the most remarkable features both in semi-Riemannian geometry and General Relativity \cite{MR0424186,MR719023}.  Despite of the fact that numerous aspects of General Relativity have their mathematical foundations in different geometrical properties of null submanifolds, it has only been recently that a mathematical framework for null submanifold geometry similar to its classical Riemannian counterpart was developed \cite{MR1383318,MR2735275, MR2598375}. In particular, a good amount of research has been devoted to analyze the geometric structure of null hypersurfaces $M^{n+1}$ immersed in Lorentzian manifolds $\bar{M}^{n+2}$. 
In this case, the choice of an $n$-dimensional  spacelike distribution $S(TM)$, the so-called \emph{screen distribution}, plays a fundamental role. In the physical scenario, screen distributions arise naturally as tangent spaces to spacelike foliations of event horizons, like the one generated by the surface of a collapsing star \cite{MR1172768, MR757180}.

Some null hypersurfaces $M$ admit a very special kind of screen distributions, whose geometry closely resembles the geometry of $M$. These hypersurfaces are known as \emph{screen conformal} and have been studied extensively \cite{MR2039644, MR2662971}. Intuitively, a screen conformal hypersurface inherits directly the geometric aspects of its screen distribution. For instance, we have that if $M\subset\bar{M}$ is screen conformal, then its screen distribution  is integrable and its integral manifolds are totally geodesic (or totally umbilical) as codimension 2 submanifolds of $\bar{M}$ if and only if $M$ is a  totally geodesic (or totally umbilical) null hypersurface of $\bar{M}$ (refer to  theorem 2.2.9 in \cite{MR2598375}).

Thus, screen conformal hypersurfaces are a natural class to explore when it comes to classifying null hypersurfaces satisfying relevant geometric conditions, since in this context we are able to translate the problem from a degenerate (null) setting to a simpler Riemannian (spacelike) scenario. Nevertheless, an important class of examples of interest both to physics and mathematics does not fit in this setting, namely, the class of null hypersurfaces of Generalized Robertson-Walker (GRW) spacetimes. In this work we show that such examples fall in a broader class which is called \emph{screen quasi-conformal}.

As an application of this notion we are able to establish conditions for integrability and umbilicity of the screen distribution. Furthermore,
we find Cartan identities for null screen isoparametric hypersurfaces embedded in Lorentzian space forms. From these identities --which closely resemble their semi-Riemannian counterparts-- we can recover and improve the classification results obtained by Atindogbe et al \cite{MR3270005}.  Another application we explore in this work is related to the classification of null Einstein hypersurfaces. First introduced by Duggal and Jin, these hypersurfaces were studied in depth in the screen conformal setting and sharp classification results were obtained when $\bar c=0$ \cite{MR2039644, MR2735275, MR2662971}. Here we extend those results to a broader class of null hypersurfaces in GRW spacetimes of non-vanishing curvature.

This paper is divided as follows: in section \ref{prelim} we lay out the basic theory concerning null hypersurfaces of Lorentzian manifolds and their screen distributions. Then, in section \ref{sec:screen} we establish the notion of screen quasi-conformal hypersurfaces and prove the main general results pertaining such hypersurfaces. In section \ref{sec:cartan} we use the tools developed so far to derive Cartan type formulas and provide a classification of null isoparametric hypersurfaces in Lorentzian space forms. Finally, in section \ref{sec:Einstein} we study null Einstein hypersurfaces in Lorentzian space forms.

\section{Preliminaries}\label{prelim}

Let $(\bar M^{n+2},\bar g)$ be a $(n+2)$-dimensional, semi-Riemannian manifold with metric $\bar g$ of constant index $q\in\{1,\dots,n+1\}$. A hypersurface $M$ of $\bar M$ is {\em null} if the {\em radical bundle} $\mathrm{Rad}(TM)=TM\cap TM^\perp$ is different from zero at each $p\in M$. As in \cite{MR1383318} and \cite{MR2598375}, a {\em screen distribution} $S(TM)$ on $M$ is defined as a non-degenerate vector bundle complementary to $TM^\perp$. A null hypersurface with an specific screen distribution is denoted $(M,g,S(TM))$. From \cite{MR1383318}, we know that there is a vector bundle $\mathrm{tr}(TM)$ of rank $1$ over $M$, called the {\em transversal bundle}, such that for each non-zero section $\xi\in\Gamma(TM^\perp)$ defined in an open set $U\subset M$ there is a unique section $N\in\Gamma(\mathrm{tr}(TM))$ such that
\begin{equation*}%\label{eq:propiedadesdeN}
\bar g(\xi,  N )=1, \quad
\bar g( N , N )=\bar g( N ,X)=0
\end{equation*}
for each $X\in\Gamma(S(TM\vert_{U}))$. We will work hereafter in a maximal neighbourhood $U$ with these properties and omit the reference to it. We write
\begin{equation}  \label{eq:descomposicion0}
T\bar M\vert_M=TM\oplus \mathrm{tr}(TM).
\end{equation}
and\begin{equation}  \label{eq:descomposicion1}
TM=S(TM)\oplus_{\mathrm{orth}} \mathrm{Rad}(TM),
\end{equation}
so that
\begin{equation*}
T\bar M\vert_M=S(TM)\oplus_{\mathrm{orth}}(\mathrm{Rad}(TM)\oplus\mathrm{tr}(TM)).
\end{equation*}

If $\bar\nabla$ denotes the Levi-Civita connection of $\bar M$, $P$ the projection of $\Gamma(TM)$ onto $\Gamma(S(TM))$ using the decomposition (\ref{eq:descomposicion1}),$\eta$ the $1$-form defined in $\Gamma(TM)$ by 
\begin{equation*}
\eta(X)=\bar g(X,N)
\end{equation*} 
and $\tau$ the $1$-form on $\Gamma(TM)$ given by
\begin{equation}\label{eq:tau}
\tau(X)=\bar g(\bar\nabla_X N ,\xi)=\bar g(\bar\nabla_X^t N ,\xi),
\end{equation}
then the \emph{local Gauss-Weingarten formulae} relative to the decompositions (\ref{eq:descomposicion0}) and (\ref{eq:descomposicion1}) are
\begin{equation}  \label{eq:gauss1}
\begin{array}{rcl}
\bar\nabla_{X}Y & = &  \nabla_{X}Y+h(X,Y) = \nabla_{X}Y+B(X,Y) N, \\[0.1cm]
\bar\nabla_{X} N & = & -A_{ N }X+\nabla_X^tN = -A_{ N }X+\tau(X)  N; \\[0.1cm]
\nabla_{X}PY & = & \nabla_{X}^*PY+h^*(X,PY) = \nabla_{X}^*PY+C(X,PY)\xi; \\[0.1cm]
\nabla_{X} \xi & = & -A_{\xi}^*X+\nabla_X^{*t} \xi= -A_{\xi}^*X-\tau(X) \xi .
\end{array}
\end{equation}
for all $X,Y\in\Gamma(TM)$. Here $\nabla$, $\nabla^t$, $\nabla^*$ and $\nabla^{*t}$ denote the induced connections on $TM$, $\mathrm{tr}(TM)$, $S(TM)$ and $\mathrm{Rad}(TM)$, respectively; $h$ and $h^*$ are the second fundamental forms of $M$ and $S(TM)$, while
\begin{equation}\label{eq:2ffB}
B(X,Y) = \bar g( \bar\nabla_{X}Y, \xi ) = \bar g(h(X,Y),\xi)=g( A_\xi^*X,Y),
\end{equation}
\begin{equation}\label{eq:2ffC}
C(X,PY) = \bar g( \nabla_X PY, N)= \bar g(h^*(X,PY),\xi)=g( A_NX,PY),
\end{equation}
are the \emph{local second fundamental forms} of $M$ and $S(TM)$. $A_ N $ and $A_\xi^*$ are the {\em local shape operators} of $M$ and $S(TM)$.

\begin{proposicion}\label{prop:propiedades}
The following properties hold true:
\begin{enumerate}
\item $\bar g(A_NX,N)=0$ for every $X\in\Gamma(TM)$;
\item $A_\xi^*\xi=0$;
\item $A_\xi^*$ is symmetric relative to $g$, that is,
\[
g(A_ \xi^*X,Y)=g(X,  A_\xi^*Y)
\]
for each $X,Y\in\Gamma(TM)$.
\item The following conditions are equivalent:
\begin{enumerate}
\item The screen distribution $S(TM)$ is integrable;
\item $h^*(X,Y)=h^*(Y,X)$ for any $X,Y\in\Gamma(S(TM))$;
\item $A_ N$ is symmetric on $\Gamma(S(TM))$ relative to $g$, that is,
\[
g( A_ N  X,Y) =g( X,  A_ N  Y) 
\]
for each $X,Y\in\Gamma(S(TM))$.
\end{enumerate}
\item\label{prop:propiedades:item5} For each $X,Y,Z\in\Gamma(TM)$,
\[
(\nabla_Xg)(Y,Z)=B(X,Y)\eta(Z)+B(X,Z)\eta(Y).
\]
\end{enumerate}
\end{proposicion}

For the proofs, see the references \cite{MR1383318} and \cite{MR2598375}.
\begin{definicion}\label{defi:derivadas}
If $X,Y,Z\in\Gamma(TM)$, we define
\begin{equation*}
\begin{array}{rcl}
(\nabla_Xh)(Y,Z) & = & \nabla_X^t(h(Y,Z))-h(\nabla_XY,Z)-h(Y,\nabla_XZ),\\[0.1cm]
(\nabla_Xh^*)(Y,PZ) & = & \nabla_X^{*t}(h^*(Y,PZ))-h^*(\nabla_XY,PZ)-h^*(Y,\nabla_X^*PZ),\\[0.1cm]
(\nabla_XB)(Y,Z) & = & X(B(Y,Z))-B(\nabla_XY,Z)-B(Y,\nabla_XZ),\\[0.1cm]
(\nabla_XC)(Y,PZ) & = & X(C(Y,PZ))-C(\nabla_XY,PZ)-B(Y,\nabla_X^*PZ),\\[0.1cm]
(\nabla_XA_N)Y & = & \nabla_X(A_NY)-A_N(\nabla_XY),\\[0.1cm]
(\nabla_XA_\xi^*)Y & = & \nabla_X(A_\xi^*Y)-A_\xi^*(\nabla_XY),
\end{array}
\end{equation*}
while for $X,Y\in\Gamma(S(TM))$, we define
\begin{equation*}
(\nabla_X^*A_\xi^*)Y = \nabla_X^*(A_\xi^*Y)-A_\xi^*(\nabla_X^*Y)
= \nabla_X^*(A_\xi^*Y)-A_\xi^*(\nabla_XY).
\end{equation*}
\end{definicion}

Denote by $\bar R$, $R$ and $R^*$ the curvature tensors of $\bar\nabla$, $\nabla$ and $\nabla^*$, respectively. For every $X,Y,Z\in\Gamma(TM)$ then we have (see \cite{MR1383318})
\begin{eqnarray}\label{eq:2.21Ciriaco}
\bar R(X,Y)Z =R(X,Y)Z&+&A_{h(X,Z)}Y-A_{h(Y,Z)}X\\ \nonumber
&+&(\nabla_Xh)(Y,Z)-(\nabla_Yh)(X,Z),
\end{eqnarray}
from which we obtain the following \emph{Gauss-Codazzi equations}:

\begin{eqnarray}\label{eq:2.22Ciriaco1}
g(R(X,Y)PZ,PW) = g(R^*(X,Y)PZ,PW)+C(X,PZ)B(Y,PW) \\ \nonumber
-C(Y,PZ)B(X,PW),
\end{eqnarray}
\begin{eqnarray}\label{eq:2.22Ciriaco2}
g(R(X,Y)Z,\xi) &=& g((\nabla_Xh)(Y,Z)-(\nabla_Yh)(X,Z),\xi ), \\
\bar{g}(R(X,Y)Z,N)&=& \bar{g}(\bar{R}(X,Y)Z,N).
\end{eqnarray}
Furthermore, we have
\begin{equation}\label{eq:2.23Ciriaco}
\bar g(\bar R(X,Y)\xi,N) = C(Y,A_\xi^*X)-C(X,A_\xi^*Y)-2d\tau(X,Y),
\end{equation}
where
\[
2d\tau(X,Y)=X(\tau(Y))-Y(\tau(X))-\tau([X,Y]).
\]

Let us recall that the connection $\nabla$, though not metric in general, it is always torsion free. Thus its Riemann tensor satisfies some of the expected symmetries.

\begin{proposicion}
The Riemann tensor of $\nabla$ satisfies the following relations for all $X,Y,Z\in\Gamma (TM)$:
\begin{description}
\item[1] $R(X,Y)Z=-R(Y,X)Z$,
\item[2] $R(X,Y)Z+R(Y,Z)X+R(Z,X)Y=0$, and
\item[3] $(\nabla_XR)(Y,Z)+(\nabla_YR)(Z,X)+(\nabla_ZR)(X,Y)=0$. 
\end{description}
Moreover, for all $X,Y,Z,W\in S(TM)$, we have:
\begin{description}
\item[4]
%\begin{eqnarray*}
$g(R(X,Y)Z,W) + g(R(X,Y)W,Z)= g(h(Y,Z)h^*(X,W))$  \\
$+g(h(Y,W),h^*(X,Z)) -g(h(X,W), h^*(Y,Z)) -g(h(X,Z),h^*(Y,W)),$
%\end{eqnarray*}
\item[5] $g(R(Y,W)X,Z)-g(R(X,Z)Y,W)=g(R(Z,X)W,Y)-g(R(W,Y)Z,X)$.
\end{description}
\end{proposicion}

We will be interested in ambient manifolds $\bar M$ with constant curvature, for which we have the following result and its Corollary; see the proof in \cite{MR3270005}:

\begin{proposicion}\label{lema:lema3.3ciriaco}(\cite[p. 34]{MR3270005}) 
Let $(\bar M_{\bar c}^{n+2},\bar g)$ be a semi-Riemannian manifold of constant curvature $\bar c$ and $M$ a null hypersurface of $\bar M$. For any $X,Y,Z\in\Gamma(TM)$ we have
\begin{enumerate}
\item\label{lema:lema3.3ciriaco:item1} $R(X,Y)Z=\bar c(g(Y,Z)X-g(X,Z)Y)-B(X,Z)A_NY+B(Y,Z)A_NX$;
\item $(\nabla_XB)(Y,Z)-(\nabla_YB)(X,Z)=B(X,Z)\tau(Y)-B(Y,Z)\tau(X)$;
\item $B(A_NY,X)-B(A_NX,Y)=2d\tau(X,Y)$;
\item $(\nabla_YA_N)X-(\nabla_XA_N)Y=\bar c(\eta(Y)X-\eta(X)Y)+\tau(Y)A_NX-\tau(X)A_NY$;
\item\label{item:5prop} $(\nabla_XA_\xi^*)Y-(\nabla_YA_\xi^*)X=\tau(Y)A_\xi^*X-\tau(X)A_\xi^*Y-2d\tau(X,Y)\xi$;
\item $\nabla_XPZ=\nabla_XZ-X(\eta(Z))\xi+\eta(Z)A_\xi^*X+\eta(Z)\tau(X)\xi$.
\end{enumerate}
\end{proposicion}

\begin{corolario}\label{cor:lema3.3ciriaco} 
Let $(\bar M_{\bar c}^{n+2},\bar g)$ be a semi-Riemannian manifold of constant curvature $\bar c$ and $M$ a null hypersurface of $\bar M$. For any $X,Y,Z\in\Gamma(TM)$ we have
\begin{enumerate}
\item\label{item:5prop} $(\nabla_X^*A_\xi^*)Y-(\nabla_Y^*A_\xi^*)X=\tau(Y)A_\xi^*X-\tau(X)A_\xi^*Y$;
\item $\nabla_X^*PZ=\nabla_XZ+\eta(Z)A_\xi^*X$.
\end{enumerate}
\end{corolario}

%%%%%%%%%%%%%%%%%%%
%%%%%%%%%%%%%%%%%%%
%%%%%%%%%%%%%%%%%%%      SECCION  QUASI-CONFORME
%%%%%%%%%%%%%%%%%%%

\section{Screen quasi-conformal hypersurfaces}\label{sec:screen}

Screen conformal hypersurfaces were first introduced by Atindogbe and Duggal \cite{MR2039644} in the aim for relating the geometry of the null hypersurface $(M,g)$ and that of a specially chosen screen distribution $S(TM)$. Roughly speaking, geometrical properties of its screen distribution translates to the analog properties on the null screen conformal hypersurface. By definition, a hypersurface is said to be \emph{screen conformal} if the shape operators are linearly related, that is, if $A_N=\varphi A^*_\xi$ for some smooth function $\varphi$. Some of the most remarkable null hypersurfaces, as null cones and null planes in Lorentz-Minkowski spacetime, are screen conformal \cite{MR2358723}. Here we extend the notion of conformality in order to include other relevant examples of null hypersurfaces that, though not screen conformal in general, do posses a rich geometrical structure.

\begin{definicion}\label{def:qconforme}
A null hypersurface $(M,g,S(TM))$ of a semi-Riemannian manifold is {\em locally screen quasi-conformal} if the shape operators $A_N$ and $A_\xi^*$ of $M$ and $S(TM)$ satisfy
\[
A_N=\varphi A_\xi^*+ \psi P,
\]
in $\Gamma(TM)$ for some functions $\varphi$ and $\psi$; recall $P:\Gamma(TM)\to \Gamma(S(TM))$ is the natural projection related to the decomposition (\ref{eq:descomposicion1}). If the above property holds everywhere on $M$, we say that $(M,g,S(TM))$ is {\em globally screen quasi-conformal}, or  {\em screen quasi-conformal} for short. We further refer to $(\varphi ,\psi)$ as the \emph{quasi-conformal pair} associated to $(M,g,S(TM))$.
\end{definicion}

\begin{remark}\label{rem:psinocero}
If $\psi\equiv 0$, we recover the definition of null screen conformal hypersurfaces introduced in \cite{MR2039644}. Thus, throughout this work we will be assuming that $\psi\not\equiv 0$.
\end{remark}

Equivalently,  $(M,g,S(TM))$ is  locally screen quasi-conformal if the local second fundamental forms $B$ and $C$ defined in (\ref{eq:2ffB}) and (\ref{eq:2ffC}) satisfy
\[
C(X,PY)=\varphi B(X,PY)+\psi g(X,PY)
\]
for any $X,Y\in\Gamma(TM)$.

\begin{ejemplo}\label{ejemploGRW}
Let $\bar M$ be a {\em Generalized Robertson-Walker (GRW) spacetime}; that is, a Lorentzian warped product of the form $-I\times_{\varrho} F$, where $F$ is a $(n+1)$-dimensional Riemannian manifold and $\varrho$ is a differentiable, positive function defined in a real interval $I\subset\mathbb R$. 

We recall briefly some results from \cite{MR3508919}. Let $f:F\to\mathbb R$ be a differentiable function. Then the graph of $f$ given as
\[
\{\,(f(p),p)\,\vert\,p\in F\,\}
\]
is a null hypersurface in $-I\times_{\varrho} F$ if and only if $f$ is a smooth function satisfying
\begin{equation*}%\label{eq:transnormal}
\vert \grad f\vert =\varrho\circ f;
\end{equation*}
we write $\varrho\circ f$ simply as $\varrho$. In this context, set
\[
\xi=\frac{1}{\sqrt{2}}\left(1,\frac{\grad f}{\varrho^2}\right).
\]

We denote by $S^*(TM)$ the screen distribution given as the family of tangent bundles of the level hypersurfaces $S_t=M\cap\left(\{t\}\times F\right)$ and
\[
N=\frac{1}{\sqrt{2}}\left(-1,\frac{\grad f}{\varrho^2}\right).
\]

Since $A_\xi^*\xi=0$ and $A_N\xi=0$, we have the following relation between the shape operators (see the proof of Proposition 5.2 in \cite{MR3508919}):
\begin{equation*}%\label{eq:shape}
\frac{1}{\sqrt{2}}(A_N-A_\xi^*)=\frac{\varrho'}{\varrho}P.
\end{equation*}
 Therefore, every null hypersurface in $-I\times_\varrho F$ is screen quasi-conformal, but not screen conformal in general.
\end{ejemplo}

We now extend some results well known for screen conformal hypersurfaces to our case. The first result gives us a criterion under which a null hypersurface is locally screen quasi-conformal. The reader may compare this with Proposition 2.2.2 in  \cite[p. 53]{MR2598375}.

\begin{proposicion}
Let $(M,g,S(TM))$ be a null hypersurface of a semi-Riemann\-ian manifold $(\bar M,\bar g)$. Suppose that $S(TM)$ is integrable and that each leaf $M'$ of $S(TM)$ is a codimension $2$ non-degenerate submanifold of $\bar M$ which is $\eta$-totally umbilical for some nowhere vanishing spacelike vector field $\eta$ normal to $M'$. If the screen distribution is parallel along integral curves of the radical distribution, then $M$ is screen quasi-conformal.
\end{proposicion}

\begin{proof}
Let $M'$ be a leaf of $S(TM)$ and write $\eta=\alpha\xi+\beta N$. For $X\in\Gamma(TM')$ we have
\begin{eqnarray*}
\bar\nabla_X\eta & = & X(\alpha)\xi+\alpha\bar\nabla_X\xi+X(\beta)N+\beta\bar\nabla_XN \\
 & = & (-\alpha A_\xi^*X-\beta A_NX)+(X(\alpha)\xi-\alpha\tau(X)\xi+X(\beta)N+\beta\tau(X)N).
\end{eqnarray*}

By equating the parts tangent to $M'$,
\[
A_\eta X=\alpha A_\xi^*X+\beta A_NX.
\]
Since $M'$ is $\eta$-totally umbilical, $A_\eta=\gamma I$ for some function $\gamma$. On the other hand, since $\eta$ is spacelike, either $\alpha\ne 0$ or $\beta\neq 0$. Without loss of generality we assume $\beta\neq 0$ and obtain
\[
A_NX=-\frac{\alpha}{\beta}A_\xi^*X+\frac{\gamma}{\beta}X,\quad X\in\Gamma(TM').
\]

Given that $S(TM)$ is parallel along integral curves of $\mathrm{Rad}(TM)$ implies $g(A_N\xi,X)=-\bar g(\bar\nabla_\xi N,X)=\bar g(N,\bar\nabla_\xi X)=0$; since $g(A_N\xi,N)=0$ as well, then $A_N\xi=0$ and
\[
A_NX=-\frac{\alpha}{\beta}A_\xi^*X+\frac{\gamma}{\beta}PX
\]
holds true everywhere in $\Gamma(TM)$; hence, $M$ is screen quasi-conformal.
\end{proof}

We also have a relation between the behavior of the principal curvatures of $M$ and the fact of $M$ being locally screen quasi-conformal; see Theorem 2.2.4 as its conformal analog in \cite[p. 54]{MR2598375}.

\begin{teorema}\label{teo:curvatures}
Let $(M,g,S(TM))$ be a null hypersurface in a semi-Riemannian manifold $(\bar M,\bar g)$. $(M,g,S(TM))$ is locally screen quasi-conformal in some domain $U$ in $M$ if and only if the following two conditions hold:
\begin{enumerate}
\item The shape operators $A_N$ and $A_\xi^*$ commute on $U$;

\item If $\mu_i$ and $\lambda_i$ are the corresponding principal curvatures of $A_N$ and $A_\xi^*$, then (reordering if necessary) $\xi$ is an eigenvector of both operators with $\mu_0=\lambda_0=0$, while for $i=1,\dots,n$ we have $\mu_i=\varphi\lambda_i+\psi$  for some differentiable functions $\varphi,\psi$ on $U$.
\end{enumerate}
\end{teorema}

\begin{proof}
Suppose first that $(M,g,S(TM))$ is locally screen quasi-conformal in $U$, so $A_N=\varphi A_\xi^*+\psi P$ for some functions $\varphi,\psi$ defined on $U$. Since it is always true that $A_\xi^*\xi=P\xi=0$, we have that $\xi$ is an eigenvector of $A_N$ with $\mu_0=0$.

It is also clear from the quasi-conformality condition that $A_N$ and $A_\xi^*$ commute; therefore, they are simultaneously diagonalizable and there exists a frame $\{ E_0=\xi,E_1,\dots,E_n\}$ such that each $E_i$ is a field of eigenvectors of $A_N$ and $A_\xi^*$ with eigenvalues $\mu_i$ and $\lambda_i$, respectively. For $i=1,\dots,n$, write $E_i=\bar E_i+c_i\xi$, where $\bar E_i\in\Gamma(S(TM))$. Note that $\bar E_i\ne 0$. Since $A_NE_i=\mu_i E_i$ and $A_\xi^*E_i=\lambda_iE_i$, we have
\begin{equation*}%\label{eq:componentes}
\mu_iE_i=A_NE_i=\varphi A_\xi^*E_i+\psi P E_i=\varphi \lambda_i E_i+\psi PE_i;
\end{equation*}
by taking the component relative to $S(TM)$, we have $\mu_i=\varphi\lambda_i+\psi$.

Conversely, if $A_N$ and $A_\xi^*$ commute on $U$; then these operators are simultaneously diagonalizable and there exists a frame $\{E_0,E_1,\dots,E_n\}$ with corresponding eigenvalues $\mu_i$ and $\lambda_i$ satisfying $\mu_0=\lambda_0=0$ and $\mu_i=\varphi\lambda_i+\psi$ for $i=1,\dots,n$. By our hypothesis, we may suppose that $E_0=\xi$, so that
\[
A_NE_0=(\varphi A_\xi^*+\psi P)E_0=0.
\]

On the other hand, for $i=1,\dots,n$, write $E_i=\bar E_i+c_i\xi$, where $\bar E_i\in\Gamma(S(TM))$. We have
\[
A_NE_i=\mu_iE_i=(\varphi \lambda_i+\psi) E_i=\varphi A_\xi^*E_i+\psi E_i=\varphi A_\xi^*E_i+\psi(\bar E_i+c_i\xi);
\]
since $A_N$ and $A_\xi^*$ are $S(TM)$-valued, $\psi c_i=0$ and we have two cases: either $\psi=0$ and $(M,g,S(TM))$ is locally screen conformal, so the result follows from \cite{MR2598375}; or $\psi\ne 0$, from which we have $c_i=0$ for all $i$, meaning that $E_i\in\Gamma(S(TM))$ and hence $PE_i=E_i$. In short,
\[
A_N=\varphi A_\xi^*+\psi P
\]
holds for the frame $\{E_0,\dots,E_n\}$ and therefore it holds everywhere on $U$. \end{proof}

The following result shows that the screen quasi-conformal hypersurfaces exhibit a similar geometry to that of their screen distribution. Recall that $M$ is said to be {\em totally umbilical} in $\bar M$ if and only if $B(X,Y)=\beta g(X,Y)$ for any $X,Y\in\Gamma(TM)$; in the case $\beta\equiv 0$, $M$ is {\em totally geodesic}. Also, compare the resut with  Theorem 2 in \cite{MR2039644}.

\begin{teorema}\label{teo:umbilica}
Let $(M,g,S(TM))$ be a locally screen quasi-conformal hypersurface of a semi-Riemannian manifold $(\bar M,\bar g)$. Then the screen distribution $S(TM)$ is integrable. Moreover, $M$ is totally umbilical in $\bar M$ if and only if each leaf $M'$ of $S(TM)$ is a totally umbilical codimension $2$ non-degenerate submanifold of $\bar M$.
\end{teorema}

\begin{proof}
Suppose $A_N=\varphi A_\xi^*+ \psi P$; since $A_\xi^*$ and $P$ are symmetric, the same happens with $A_N$. Notice that item 2.1 in Proposition \ref{prop:propiedades} implies that $S(TM)$ is integrable.

Now, let $M'$ be a leaf of $S(TM)$. If $\nabla^*$ denotes its Levi-Civita connection and $h'$ its second fundamental form, then
\[
\bar \nabla_XY=\nabla_X^*Y+h'(X,Y),
\]
for each $X,Y\in\Gamma(TM')$; therefore,
\begin{eqnarray*}
h'(X,Y) & = & C(X,Y)\xi+B(X,Y)N = g(A_NX,Y)\xi+g(A_\xi^*X,Y)N \\
            & = & g((\varphi A_\xi^*+\psi)X,Y)\xi+g(A_\xi^*X,Y)N \\
            & = & B(X,Y)(\varphi\xi+N)+\psi g(X,Y)\xi;
\end{eqnarray*}

If $M$ is totally umbilical in $\bar M$, say $B(X,Y)=\beta g(X,Y)$, we have
\[
h'(X,Y) = g(X,Y)((\beta\varphi+\psi)\xi+\beta N)
\]
and $M'$ is totally umbilical. Conversely, if $M'$ is totally umbilical, then $h'(X,Y)=g(X,Y)H$, where $H$ is a vector field normal to $M'$ which can be written as $H=\alpha\xi+\beta N$. From the above relation,
\[
g(X,Y)(\alpha\xi+\beta N)= B(X,Y)(\varphi\xi+N)+\psi g(X,Y)\xi;
\]
by taking the component relative to $N$, we have $B(X,Y)=\beta g(X,Y)$ for $X,Y\in\Gamma(S(TM))$. Further $B(X,\xi)=\beta g(X,\xi)=0$ for any $X\in\Gamma(S(TM))$, so  $B(X,Y)=\beta g(X,Y)$ for $X,Y\in\Gamma(TM)$ and $M$ is totally umbilical in $\bar M$.
\end{proof}

%%%%%%%%%%%%%%%%%%%
%%%%%%%%%%%%%%%%%%%
%%%%%%%%%%%%%%%%%%%      SECCION  CARTAN
%%%%%%%%%%%%%%%%%%%

\section{Cartan identities for null screen isoparametric hypersurfaces}\label{sec:cartan}

As a first application of the notion of quasi-conformality we proceed to elaborate on the notion of null screen isoparametric hypersurfaces and establish algebraic identities of Cartan type. The study of isoparametric hypersurfaces in the Riemannian setting goes back to the early works of Cartan \cite{MR1553310,MR0000169} and is to this day a vivid area of research (refer to \cite{MR3408101} and references therein for an updated account). These hypersurfaces can be defined as having constant principal curvatures. Thus, following \cite{MR3270005} we state the next definition:

Let $(M,g,S(TM))$ be a null hypersurface of a Lorentzian manifold of constant curvature $(\bar M_{\bar c}^{n+2},\bar g)$. Since the shape operator $A_\xi^*$ is diagonalizable and $A_\xi^*\xi=0$, we have a frame field $\{\xi,E_1,\dots,E_n\}$ of eigenvectors of $A_\xi^*$ such that $\{E_1,\dots,E_n\}$ is an orthonormal frame field of $S(TM)$.  If $A_\xi^*E_i=\lambda_i E_i$, $i=1,\dots,n$, we call $\lambda_i$ the {\em screen principal curvatures} of $(M,g,S(TM))$.

\begin{definicion} \label{def:taus}
Let $(\bar M_{\bar c}^{n+2},\bar g)$ be a Lorentzian manifold of constant curvature $\bar c$ and $(M,g,S(TM))$ a null hypersurface of $\bar M$. $(M,g,S(TM))$ is a {\em null screen isoparametric} hypersurface if all the screen principal curvatures are constant along $S(TM)$. For each screen principal curvature $\lambda$, we define the distribution
\[
T_\lambda =\{\ X\in\Gamma(S(TM))\ \vert\ A_\xi^*X=\lambda X\ \}.
\]
\end{definicion}

The following result is essentially Lemma 3.3 in \cite{MR3270005}, with some additions.

\begin{lema}\label{lema:Codazzi} Let $(\bar M_{\bar c}^{n+2},\bar g)$ be a Lorentzian manifold of constant curvature $\bar c$ and $(M,g,S(TM))$ a null hypersurface of $\bar M$ such that the $1$-form $\tau$ given in (\ref{eq:tau}) vanishes along $S(TM)$. Then
\begin{enumerate}
\item\label{item:1lema} For all $X,Y\in\Gamma(S(TM))$,
\begin{enumerate}
\item $(\nabla_X A_\xi^*)Y=(\nabla_Y A_\xi^*)X;$
\item $(\nabla_X^* A_\xi^*)Y=(\nabla_Y^* A_\xi^*)X;$
\item $(\nabla_X A_\xi^*)\xi=(\nabla_\xi A_\xi^*)X+\tau(\xi)A_\xi^*X-C(\xi,A_\xi^*X)\xi ;$ and
\item $(\nabla_X^* A_\xi^*)\xi=(\nabla_\xi^* A_\xi^*)X+\tau(\xi)A_\xi^*X.$
\end{enumerate}
\smallskip

\item\label{item:2lema} $\nabla_XA_\xi^*$ and $\nabla_X^*A_\xi^*$ are symmetric relative to $g$; i.e., for any $X,Y,Z\in\Gamma(TM)$ we have
\begin{enumerate}
\item $g((\nabla_XA_\xi^*)Y,Z) = g(Y,(\nabla_XA_\xi^*)Z)$; and 
\item $g((\nabla_X^*A_\xi^*)Y,Z) = g(Y,(\nabla_X^*A_\xi^*)Z).$
\end{enumerate}
\smallskip

\item\label{item:3lema} For any $X,Y,Z\in\Gamma(S(TM))$,
\begin{enumerate}
\item $g((\nabla_XA_\xi^*)Y,Z)=g((\nabla_ZA_\xi^*)Y,X)$;
\item $g((\nabla_X^*A_\xi^*)Y,Z)=g((\nabla_Z^*A_\xi^*)Y,X)$;
\item $g((\nabla_XA_\xi^*)\xi,Z)=g((\nabla_ZA_\xi^*)\xi,X);$ and
\item $g((\nabla_\xi A_\xi^*)Y,Z)=-\tau(\xi)g(A_\xi^*Y,Z)$.
\end{enumerate}
\smallskip

\item \label{item:4lema} For $X\in\Gamma(TM)$, $Y\in T_\lambda$ and $Z\in T_\mu$, $\lambda\ne\mu$,
\begin{enumerate}
\item $g((\nabla_XA_\xi^*)Y,Z)=(\lambda-\mu)g(\nabla_XY,Z)$.
\item $g((\nabla_X^*A_\xi^*)Y,Z)=(\lambda-\mu)g(\nabla_X^*Y,Z)$.
\end{enumerate}
\end{enumerate}
\end{lema}

\begin{proof} The proofs of items (a) were given in \cite{MR3270005} and we will omit them. We note that item \ref{item:2lema}(a) was stated in the cited reference only for $Y,Z\in\Gamma(S(TM))$, although the proof given works for any $Y,Z\in\Gamma(TM)$. Items (b) follow from items (a) by taking the projection $P:\Gamma(TM)\to\Gamma(S(TM))$.

To prove item \ref{item:1lema}(c), we use Proposition \ref{lema:lema3.3ciriaco}:
\begin{eqnarray*}
(\nabla_XA_\xi^*)\xi-(\nabla_\xi A_\xi^*)X & = & \tau(\xi)A_\xi^*X-\tau(X)A_\xi^*\xi-2d\tau(X,\xi)\xi \\
& = & \tau(\xi)A_\xi^*X-2d\tau(X,\xi)\xi.
\end{eqnarray*}

Since $\bar M$ has constant curvature, $\bar R(X,\xi)\xi=0$; by (\ref{eq:2.23Ciriaco}) we have
\[
0=\bar g(\bar R(X,\xi)\xi,N)=C(\xi,A_\xi^*X)-C(X,A_\xi^*\xi)-2d\tau(X,\xi),
\]
and by substituting into the above we obtain item \ref{item:1lema}(c). By projecting the expression in item \ref{item:1lema}(c) we obtain that of item \ref{item:1lema}(d).

To prove item \ref{item:3lema}(c), we have
\begin{eqnarray*}
g((\nabla_XA_\xi^*)\xi,Z) & = & g((\nabla_\xi A_\xi^*)X+\tau(\xi)A_\xi^*X-C(\xi,A_\xi^*X)\xi,Z) \\
& = & g((\nabla_\xi A_\xi^*)Z,X)+\tau(\xi)g(A_\xi^*Z,X) \\
& = & g((\nabla_\xi A_\xi^*)Z+\tau(\xi)A_\xi^*Z-C(\xi,A_\xi^*Z)\xi,X) \\
& = & g((\nabla_ZA_\xi^*)\xi,X).
\end{eqnarray*}

The proof of item \ref{item:3lema}(d) is analogous.
\end{proof}

The following is a refinement of Lemma 3.4 in \cite{MR2039644}.

\begin{lema}\label{lema:distribuciones} Let $(\bar M_{\bar c}^{n+2},\bar g)$ be a Lorentzian manifold of constant curvature $\bar c$ and $(M,g,S(TM))$ a null hypersurface of $\bar M$ such that the $1$-form $\tau$ given in (\ref{eq:tau}) vanishes along $S(TM)$. Let $\lambda,\mu$ be distinct screen principal curvatures.
\begin{enumerate}
\item If $X,Y\in T_\lambda$, then $\nabla_X^*Y\in T_\lambda$.
\item If $X\in T_\lambda$ and $Y\in T_\mu$ then $\nabla_X^*Y\perp T_\lambda$.
\end{enumerate}
\end{lema}

\begin{proof} We just observe that the proof of the first item given in \cite{MR2039644} can be shortened by using the fact that $g$ restricted to $S(TM)$ is non-degenerate. If $Z\in\Gamma(S(TM))$, then
\begin{eqnarray*}
g(A_\xi^*(\nabla_XY),Z) & = & g(\nabla_X^*(A_\xi^*Y)-(\nabla_X^*A_\xi^*)Y,Z) \\
 & = & g(\lambda\nabla_X^*Y,Z)-g((\nabla_X^*A_\xi^*)Y,Z) \\
 & = & g(\lambda\nabla_X^*Y,Z)-g(Y,(\nabla_Z^*A_\xi^*)X) \\
 & = & g(\lambda\nabla_X^*Y,Z)-g(Y,\nabla_Z^*(A_\xi^*X)-A_\xi(\nabla_Z^*X)) \\
 & = & g(\lambda\nabla_X^*Y,Z)-\lambda g(Y,\nabla_Z^*X)-g(A_\xi Y,\nabla_Z^*X) \\
 & = & g(\lambda\nabla_X^*Y,Z),
\end{eqnarray*}
implying directly that $A_\xi^*(\nabla_XY)=\lambda\nabla_X^*Y$.
\end{proof}

The following result extends Cartan's identities to null screen isoparametric hypersurfaces of Lorentzian space forms.

\begin{teorema}\label{teo:Cartan}
Let $(M,g,S(TM))$ be a null screen isoparametric hypersurface of a Lorentzian manifold $(\bar M_{\bar c}^{n+2},\bar g)$ of constant curvature $\bar{c}$ such that $\tau(X)=0$ for all $X\in S(TM)$. Let $X$ be a unit eigenvector of $A_\xi^*$ at a point $p$ and $\lambda$ the associated screen principal curvature. For any frame field $\{\xi,E_1,\dots,E_n\}$ of eigenvectors of $A_\xi^*$ such that $\{E_1,\dots,E_n\}$ is an orthonormal frame field of $S(TM)$ satisfying $A_\xi^* E_j=\lambda_j E_j $, we have
\begin{equation}\label{eq:Cartan}
 \sum_{\lambda_j \neq \lambda}\frac{\bar{c}+\lambda g(A_N E_j, E_j) + \lambda_j g(A_N X, X)}{\lambda-\lambda_j}=0.
\end{equation}
Moreover, if $(M,g,S(TM))$ is screen quasi-conformal near $p$ with quasi-conformal pair  $(\varphi$, $\psi )$, and $l>1$ distinct screen principal curvatures $\lambda_1, \ldots, \lambda_l$ with multiplicities $m_1, \ldots, m_l$, then for each $i$, $1\le i\le l$,
\begin{equation}\label{eq:Cartan conforme}
 \sum_{j \neq i} m_j \frac{\bar{c} + 2 \varphi \lambda_i \lambda_j + \psi(\lambda_i+\lambda_j)}{\lambda_i - \lambda_j}=0.
\end{equation}
\end{teorema}

\begin{proof} Let us first assume that the number $l$ of distinct screen principal curvatures is greater than two. Let $Y$ be a unit eigenvector of $A_\xi^*$ at $p$ with associated screen principal curvature $\mu \neq \lambda$. Extend $X$ and $Y$ to be eigenvector fields of $A_\xi^*$ near $p$. For clarity, we divide the proof in six steps.
\medskip

I. Using Lemma \ref{lema:Codazzi}, by direct substitution we have
\begin{eqnarray*}
g((\nabla_{[X,Y]}A_\xi^*)X,Y)&=&g((\nabla_XA_\xi^*)[X,Y],Y) \\
&=&g([X,Y],(\nabla_X A_\xi^*)Y) \\
&=&g([X,Y], (\nabla_Y A_\xi^*)X).
\end{eqnarray*}

Now,
\begin{eqnarray*}
g(\nabla_X Y, (\nabla_Y A_\xi^*)X)&=&g(\nabla_X Y, \nabla_Y(A_\xi^*X) - A_\xi^*(\nabla_Y X)) \\
&=&g(\nabla_X Y, (\lambda I - A_\xi^*)(\nabla_Y X)),
\end{eqnarray*}
and
\begin{eqnarray*}
g(\nabla_Y X, (\nabla_Y A_\xi^*)X)&=&g(\nabla_Y X, (\nabla_X A_\xi^*)Y) \\
&=&g(\nabla_Y X, \nabla_X(A_\xi^*Y) - A_\xi^*(\nabla_X Y)) \\
&=&g(\nabla_Y X, (\mu I - A_\xi^*)\nabla_X Y).
\end{eqnarray*}

Since $\nabla$ is a torsion free linear connection, we get
\begin{equation*}
g([X,Y], (\nabla_Y A_\xi^*)X)=g(\nabla_X Y, (\lambda I - A_\xi^*)(\nabla_Y X))-g(\nabla_Y X, (\mu I - A_\xi^*)\nabla_X Y), \\
\end{equation*}
from which it follows that
\begin{equation}\label{eq:paso1}
g((\nabla_{[X,Y]}A_\xi^*)X, Y)=(\lambda - \mu)g(\nabla_X Y, \nabla_Y X).
\end{equation}

\medskip

II. Using Proposition \ref{lema:lema3.3ciriaco}, item \ref{lema:lema3.3ciriaco:item1}, we have
\begin{equation*}
R(X,Y)Y=\bar{c}(g(Y,Y)X-g(X,Y)Y)-B(X,Y)A_NY+B(Y,Y)A_NX
\end{equation*}
and by equations (\ref{eq:2ffB}) and (\ref{eq:2ffC}),
\begin{eqnarray*}
g(R(X,Y)Y,X)&=&\bar{c}(g(Y,Y)g(X,X)-g(X,Y)g(X,Y)) \\
&& -g(A_\xi^*X,Y)g(A_NY,X)+g(A_\xi^*Y,Y)g(A_NX,X).
\end{eqnarray*}
Recalling that $X,Y$ are unit vectors, $X\in T_{\lambda}$, $Y\in T_{\mu}$ and $\lambda \neq \mu$ and the fact that the metric in $S(TM)$ is Riemannian,
\begin{equation}\label{eq:paso2}
g(R(X,Y)Y,X)=\bar{c}+\mu g(A_NX,X).
\end{equation}

\medskip

III. Using item \ref{prop:propiedades:item5} of Proposition \ref{prop:propiedades}, and equations (\ref{eq:2ffB}) and (\ref{eq:2ffC}) again,
\begin{eqnarray*}
(\nabla_Xg)(\nabla_YY,X) & = & B(X,\nabla_YY)\bar{g}(X,N)+B(X,X)\bar{g}(\nabla_YY,N) \\
 & = & \lambda g(A_NY,Y);
\end{eqnarray*}
since $\nabla_YY\in T_\mu$, we have $g(\nabla_YY,X)=0$ and
\begin{equation*}
\lambda g(A_NY,Y)=-g(\nabla_X\nabla_YY,X)-g(\nabla_YY, \nabla_XX).
\end{equation*}
Now, by Gauss-Weingarten equations (\ref{eq:gauss1}) and Lemma \ref{lema:distribuciones},
%\begin{eqnarray*}
%g(\nabla_YY,\nabla_XX)&=&g(\nabla_Y^*Y, \nabla_X^*X)+g(\nabla_Y^*Y, C(X,X)\xi)+g(\nabla_X^*X, C(Y,Y)\xi),\\
%&=&0.
%\end{eqnarray*}
\[
g(\nabla_YY,\nabla_XX)=g(\nabla_Y^*Y, \nabla_X^*X)+g(\nabla_Y^*Y, C(X,X)\xi)+g(\nabla_X^*X, C(Y,Y)\xi)=0.
\]

Therefore
\begin{equation}\label{eq:Rxy}
\lambda g(A_NY,Y)=-g(\nabla_X\nabla_YY,X)
\end{equation}
Similarly, we have
\[
(\nabla_Yg)(\nabla_XY,X)=B(Y,\nabla_XY)\bar{g}(X,N)+B(Y,X)\bar{g}(\nabla_XY,N)=0.
\]
%\begin{eqnarray*}
%(\nabla_Yg)(\nabla_XY,X)&=&B(Y,\nabla_XY)\bar{g}(X,N)+B(Y,X)\bar{g}(\nabla_XY,N),\\
%&=&0,
%\end{eqnarray*}
from which we get, because $\nabla_XY\perp T_\lambda$, 
\begin{equation*}
0=-g(\nabla_Y\nabla_XY,X)-g(\nabla_XY, \nabla_YX),
\end{equation*}
or
\begin{equation}\label{eq:Ryx}
g(\nabla_XY, \nabla_YX)=-g(\nabla_Y\nabla_XY, X).
\end{equation}
On the other hand, by Lemma \ref{lema:Codazzi}, item \ref{item:4lema}(a),
\begin{equation*}
g((\nabla_{[X,Y]}A_\xi^*)X,Y)=(\lambda-\mu)g(\nabla_{[X,Y]}X,Y),
\end{equation*}
which can be written, interchanging $X$ and $Y$, as follows
\begin{equation*}
g((\nabla_{[Y,X]}A_\xi^*)Y,X)=(\mu-\lambda)g(\nabla_{[Y,X]}Y,X),
\end{equation*}
and, by antisymmetry of $[X,Y]$ and Lemma \ref{lema:Codazzi}, item \ref{item:2lema}(a),
\begin{equation}\label{eq:Rxyyx}
-g((\nabla_{[X,Y]}A_\xi^*)X,Y)=(\lambda-\mu)g(\nabla_{[X,Y]}Y,X).
\end{equation}
Using equations (\ref{eq:Rxy}), (\ref{eq:Ryx}) and (\ref{eq:Rxyyx}) we obtain
\begin{multline*}
g(R(X,Y)Y,X)=g(\nabla_X\nabla_YY,X)-g(\nabla_Y\nabla_XY,X)-g(\nabla_{[X,Y]}Y,X) \\
=-\lambda g(A_NY,Y)+g(\nabla_XY, \nabla_YX)+ \frac{1}{\lambda-\mu}g((\nabla_{[X,Y]}A_\xi^*)X,Y).
\end{multline*}
Using this and equation (\ref{eq:paso1}) of step I, we get
\begin{equation*}
g(R(X,Y)Y,X)=-\lambda g(A_NY,Y)+2g(\nabla_XY, \nabla_YX).
\end{equation*}
Combining this with equation (\ref{eq:paso2}) of step II, it turns out that
\begin{equation}\label{eq:paso3}
\bar{c} + \mu g(A_NX,X) + \lambda g(A_NY,Y)= 2g(\nabla_XY, \nabla_YX).
\end{equation}

\medskip

IV. This is the step which requires at least three distinct screen principal curvatures. Let $Z$ be a unit eigenvector of $A_\xi^*$ with associated screen principal curvature $\nu\ne\lambda,\mu$; then since $Y\in T_\mu$ and $X\in T_\lambda$ with $\lambda \neq \mu$,  Lemma \ref{lema:Codazzi} thus implies
\[
g((\nabla_ZA_\xi^*)X,Y)=g((\nabla_XA_\xi^*)Z,Y) = g(Z, (\nabla_XA_\xi^*)Y) =(\mu-\nu)g(\nabla_XY,Z).
\]
%\begin{eqnarray*}
%g((\nabla_ZA_\xi^*)X,Y)&=&g((\nabla_XA_\xi^*)Z,Y),\\
%&=&g(Z, (\nabla_XA_\xi^*)Y),\\
%&=&(\mu-\nu)g(\nabla_XY,Z).
%\end{eqnarray*}

Similarly, with $X$ and $Y$ interchanged,
\begin{equation*}
g((\nabla_Y A_\xi^*)X, Z)=(\lambda-\nu)g(\nabla_Y X, Z).
\end{equation*}

Finally, by item 3(a) of Lemma \ref{lema:Codazzi}  
\begin{equation*}
g((\nabla_Z A_\xi^*)X, Y) = g((\nabla_Y A_\xi^*)X, Z)
\end{equation*}

Hence
\begin{equation}\label{eq:paso4}
(\lambda-\nu)(\mu-\nu)g(\nabla_X Y, Z) g(\nabla_Y X, Z)=(g((\nabla_Z A_\xi^*)X, Y))^2.
\end{equation}

\medskip

V. Here we express $g(\nabla_XY, \nabla_YX)$ in terms of the basis $\{\xi, E_1,\ldots,E_n\}$, using that $\nabla_XY\perp T_\lambda$ and $\nabla_YX\perp T_\mu$. We write
\begin{eqnarray*}
\nabla_XY&=&\sum_{i=1}^n g(\nabla_XY,E_i)E_i + \eta(\nabla_XY)\xi,\\
\nabla_YX&=&\sum_{j=1}^n g(\nabla_YX,E_j)E_j + \eta(\nabla_YX)\xi.
\end{eqnarray*}
Then
\begin{equation}\label{eq:paso5}
g(\nabla_XY, \nabla_YX)=\sum_{\lambda_i\neq \lambda, \mu}g(\nabla_XY,E_i)g(\nabla_YX,E_i),
\end{equation}
where $A_\xi^*E_i=\lambda_i E_i$ for $i=1,\ldots,n$.

\medskip

VI. Using equation (\ref{eq:paso4}), summing over all $\lambda_i\ne\lambda,\mu$, and by (\ref{eq:paso5}),
\begin{equation*}
\sum_{\lambda_i\neq \lambda, \mu}g(\nabla_XY,E_i)g(\nabla_YX,E_i)=\sum_{\lambda_i\neq \lambda, \mu}\frac{(g((\nabla_{E_i}A_\xi^*)X,Y))^2}{(\lambda-\lambda_i)(\mu-\lambda_i)} = g(\nabla_XY,\nabla_YX).
\end{equation*}
%This and equation (\ref{eq:paso5}) gives
%\begin{equation*}
%g(\nabla_XY,\nabla_YX)=\sum_{\lambda_i\neq \lambda, \mu}\frac{g((\nabla_{E_i}A_\xi^*)X,Y)^2}{(\lambda-\lambda_i)(\mu-\lambda_i)}.
%\end{equation*}
Now, by equation (\ref{eq:paso3}) of step III this becomes
\begin{equation*}
\bar{c}+\lambda g(A_NY,Y)+\mu g(A_NX,X)=2\sum_{\lambda_i\neq \lambda, \mu}\frac{(g((\nabla_{E_i}A_\xi^*)X,Y))^2}{(\lambda-\lambda_i)(\mu-\lambda_i)}.
\end{equation*}

Setting $Y=E_j$ (hence $\mu=\lambda_j\neq \lambda$), dividing by $\lambda-\lambda_j$, and summing over all $\lambda_j$ with this property,
\begin{equation*}
\sum_{\lambda_j\neq \lambda}\frac{\bar{c}+\lambda g(A_NE_j,E_j)+\lambda_j g(A_NX,X)}{\lambda-\lambda_j}=2\sum_{\substack{\lambda_i\neq \lambda, \lambda_j \\ \lambda_j\neq \lambda}}\frac{(g((\nabla_{E_i}A_\xi^*)X,E_j))^2}{(\lambda-\lambda_j)(\lambda-\lambda_i)(\lambda_j-\lambda_i)}.
\end{equation*}

Since the expression in the right hand side is skew-symmetric in $\{i,j\}$, its value is zero, and so the sum in the left hand side vanishes.

In the case $l=2$, let $\lambda$ and $\mu$ be the two distinct screen principal curvatures. Then, for unit eigenvectors $X\in T_\lambda$ and $Y\in T_\mu$ of $A_\xi^*$ it is enough to see that 
\begin{equation*}
\bar{c}+\lambda g(A_NY,Y)+\mu g(A_NX,X)=0,
\end{equation*}
which is a consequence of equation (\ref{eq:paso3}) and the fact that $g(\nabla_XY,\nabla_YX)=0$ by item 2 of Lemma \ref{lema:distribuciones} --since $\lambda$ and $\mu$ correspond to distinct (orthogonal) distributions $T_\lambda$ and $T_\mu$ on $S(TM)$ and $\nabla_X^*Y=\nabla_XY$ for $X,Y\in \Gamma (S(TM))$. Finally, equation (\ref{eq:Cartan conforme}) follows from Definition \ref{def:qconforme}.
\end{proof}

\begin{remark}
From equation (\ref{eq:Cartan}), the formula (3.11) in \cite{MR3270005} follows by setting $X=E_i$ and $\lambda=\lambda_i$, $i=1,\ldots,n$ and summing over all $i$. We observe that the proof of the Cartan identities given in \cite{MR3270005} is valid only for two-dimensional screen distributions $S(TM)$ because in their calculations the authors assumed that $\nabla_{E_i}E_j \in T_{\lambda_j}$ for any pair $\{ E_i, E_j \}$, which is guaranteed by the condition $\nabla_{E_i}E_j \perp T_{\lambda_i}$ only when $n=2$. In fact, equation (3.19) in \cite{MR3270005} is based on this assumption while our analogous equation , equation (\ref{eq:paso3}), was proved without that assumption. Also notice that  the authors assumed in their applications (see Section 4 in \cite{MR3270005}) that the number of distinct screen principal curvatures is at most two.  On the other hand, we note that this latter assumption follows as a consequence of Cartan identities (\ref{eq:Cartan conforme}), as the following Corollary states (see \cite{2017arXiv171107978N} for a detailed proof).
 \end{remark}

\begin{corolario}\label{class01}
Let $(M,g,S(TM))$ be a null screen isoparametric hypersurface of $(\bar M_{\bar c}^{n+2},\bar{g})$, $\bar c=0,-1$. If $\tau (X)=0$ for all $X\in \Gamma (S(TM))$, then the number $l$ of distinct screen principal curvatures of $M$ is at most $2$. If $l=2$ and $\bar c=0$, one of the screen principal curvatures is zero.
\end{corolario}

\begin{ejemplo} All Lorentzian space forms may be expressed as a GRW spacetime. Then by choosing the screen distribution $S^*(TM)$ as in example \ref{ejemploGRW}, the conditions of Theorem \ref{teo:Cartan} hold and hence the Cartan identities are valid for all null hypersurfaces of the form $(M,g,S^*(TM))$ immersed in a Lorentzian space form.
\end{ejemplo}

%%%%%%%%%%%%%%%%%%%
%%%%%%%%%%%%%%%%%%%
%%%%%%%%%%%%%%%%%%%      SECCION  EINSTEIN
%%%%%%%%%%%%%%%%%%%

\section{Null Einstein hypersurfaces}\label{sec:Einstein}

In \cite{MR2735275} Duggal and Jin classified null Einstein hypersurfaces immersed in Lorentzian space forms. We notice that their key results are restricted to the screen homothetic case. This strong assumption forces $\bar{c}=0$, so null screen homothetic hypersurfaces can only exist in Lorentz-Minkowski space.

In this section we extend some of the results given in \cite{MR2735275} to the case ${\bar c}\neq 0$ in the screen quasi-conformal setting. The results will follow as a consequence of the general results pertaining null screen quasi-conformal hypersurfaces.

First of all, let us recall that since $\nabla$ is not a metric connection the Ricci tensor associated to the Riemann endomorphism is not symmetric in general, and thus it lacks geometric significance. Nevertheless, null screen quasi-conformal hypersurfaces in semi-Riemannian space forms do admit an induced symmetric Ricci tensor.

\begin{definicion}
Let $(M,g,S(TM))$ be a null hypersurface immersed in a semi-Riemannian manifold $(\bar{M},\bar{g})$. We denote the trace
\[
\{Z\mapsto R(X,Z)Y\} \ \  \text{by}  \ \ R^{(0,2)}(X,Y). 
\]
In case $R^{(0,2)}\colon \Gamma (TM)\times \Gamma (TM)\to \mathbb{R}$ is symmetric, the tensor  thus defined is called the 
{\it induced Ricci curvature} of $M$ and is denoted by $\text{Ric}$.
\end{definicion}

\begin{proposicion}\label{prop:Ricci}
Let $(M,g,S(TM))$ be a null screen quasi-conformal hypersurface immersed in a semi-Riemannian space form $(\bar{M}_{\bar c}^{n+2},\bar{g})$ of constant sectional curvature $\bar c$. Then $M$ admits an induced symmetric Ricci tensor $\text{Ric}$.
\end{proposicion}

\begin{proof}
Let us recall that the Ricci tensors of $M$ and $\bar{M}$ are related by (refer to \cite{MR2598375}, p. 69)
\begin{equation}\label{eq:Ric02}
R^{(0,2)}(X,Y)=\bar{\text{Ric}}(X,Y)+B(X,Y)\text{tr}A_N-g(A_NX,A_\xi^*Y)-g(R(\xi ,Y)X,N)
\end{equation}
for all $X,Y\in \Gamma (TM)$. Furthermore, since $\bar{M}$ has constant sectional curvature $\bar c$ we have
\begin{eqnarray*}
g(R(\xi ,Y)X,N) &=&g(\bar{R}(\xi ,Y)X, N)\\
&=&\bar c g(g(X,Y)\xi -g(X,\xi )Y,N)\\
&=&\bar c g(X,Y).
\end{eqnarray*}
Thus, since $A_N=\varphi A_\xi^*+\psi P$ we have
\begin{eqnarray*}
R^{(0,2)}(X,Y)-R^{(0,2)}(Y,X) &=& g(A_NY,A_\xi^*X)-g(A_NX,A_\xi^*Y)\\ &=&\psi (g(PY,A_\xi^*X)-g(PX,A_\xi^*Y))\\ &=&0. \qedhere
\end{eqnarray*}
\end{proof}

%%When $M$ admits an induced symmetric Ricci tensor, it is possible to define an induced scalar curvature, in a fashion analogous to the semi-Riemannian case.

%%\begin{definicion}
%%Let $(M,S(TM),g)$ be a null hypersurface that admits a symmetric Ricci tensor $\text{Ric}$. The trace $r$ of $\text{Ric}$ is the \emph{induced scalar curvature} of $M$.
%%\end{definicion}

%%Thus, relative to a semi-orthogonal frame $\{E_1,\ldots ,E_n,\xi\}$ of $M$ we have
%%\[
%%r=\text{Ric}(\xi ,\xi )+\sum_{i=1}^n\epsilon_i\text{Ric}(E_i,E_i).
%%\] 

Moreover, due to equation (\ref{eq:Ric02}) the induced Ricci tensor in a semi-Riemannian space form of constant curvature $\bar c$  satisfies %%(refer to \cite{MR2598375})
\begin{equation}\label{eq:Ricci}
\text{Ric}\, (X,Y)=\bar c ng(X,Y)+B(X,Y)\text{tr}\, A_N-g(A_NX,A_\xi^*Y), 
\end{equation}
for all $X,Y\in \Gamma(TM)$.

In accordance to Duggal and Jin \cite{MR2735275} we define a null Einstein hypersurface as follows:

\begin{definicion}
Let $(M,g,S(TM))$ be a null hypersurface that admits a symmetric Ricci tensor $\text{Ric}$. We say that $M$ is a \emph{null Einstein hypersurface} if there exists a smooth function $k$ in $M$ which is constant along $S(TM)$ and satisfies $\text{Ric}(X,Y)=k g(X,Y)$ for all $X,Y \in \Gamma (TM)$.
\end{definicion}

It is worthwhile pointing out that --in contrast to the semi-Riemannian setting-- the well known Schur Lemma does not apply, and hence the factor $k$ may not be a constant along $M$, as it is further shown in Example \ref{ejem:hipernulas}. Since the notion of null Einstein hypersurface as first introduced by Duggal and Jin \cite{MR2735275} only considers the case in which $k$ is constant, our approach may prove useful in a broader class of situations. We show next one of such examples.

%\begin{proposicion}\label{teo:class04}
%Let $(M,g,S(TM))$ be a null screen isoparametric hypersurface with quasi-conformal pair $(\varphi ,\psi)$ immersed in a Lorentzian space form of constant curvature $\bar{M}^{n+2}_{\bar c}$. If $S(TM)$ is totally umbilical, then $(M,g,S(TM))$ is a null Einstein hypersurface.
%\end{proposicion}
%\begin{proof}
%By Theorem \ref{teo:umbilica}, $(M,g)$ admits a Ricci tensor. Let $\lambda$ be the only eigenvalue of $A_\xi^*$. Thus from equation (\ref{eq:Ricci}) we can see at once that
%\begin{equation*}
%\text{Ric} (X,Y)=(m\bar{c}+(n-1)(\varphi\lambda +\psi )\lambda )g(X,Y)
%\end{equation*}
%for all $X,Y\in\Gamma (TM)$. 
%\end{proof}

\begin{ejemplo}\label{ejem:hipernulas}
In \cite{MR3508919} the authors characterized the null totally umbilical hypersurfaces $(M,g,S^*(TM))$ of Lorentzian space forms $(\bar{M}^{n+2}_{\bar c},\bar{g})$ of non vanishing curvature and provided explicit examples with totally umbilical screen distribution $S^*(TM)$ constructed from graphs. As an illustrative example, in de Sitter space --viewed as an hyperquadric in Lorentz-Minkowski space--  consider the (signed) distance from any point $p\in \mathbb{S}^{n+1}$ to the ``parallel" given by the set of points that make a constant angle $\theta=\cos^{-1}\alpha$ with respect to the fixed canonical vector $e_{n+3}\in \mathbb{R}^{n+3}_1$. Thus we have a null hypersurface $M$ parameterized by
\begin{equation*}
\Psi (s,u_1,\ldots ,u_n)=(s,R(s)\phi (u_1,\ldots ,u_n),\sqrt{1-\alpha^2}s+\alpha )
\end{equation*}
where $\phi (u_1,\ldots ,u_n)$ is an orthogonal parametrization of the $n$-dimensional sphere and $R(s)$ is a smooth function. By means of the isometry $-\mathbb{R}\times_{\cosh} \mathbb{S}^{n+1}\to\mathbb{S}^{n+2}_1$, $(t,p)\mapsto (\sinh t, \cosh t,p)$ we find that the shape operator $A^*_\xi$ is given by
\begin{equation*}
A^*_\xi (X)=-\frac{\alpha}{\sqrt{2}(\alpha \sinh t-\sqrt{1-\alpha^2})}X.
\end{equation*}
In general, in all such examples, the eigenvalue $\lambda$ takes the form $\lambda = \lambda (t)$and thus $(M,g,S^*(TM))$ is null screen isoparametric. By Theorem \ref{teo:umbilica}, $(M,g)$ admits a Ricci tensor. Thus from equation (\ref{eq:Ricci}) we can see at once that
\begin{equation*}
\text{Ric} (X,Y)=(n\bar{c}+(n-1)(\varphi\lambda +\psi )\lambda )g(X,Y)
\end{equation*}
for all $X,Y\in\Gamma (TM)$, where $\varphi\equiv 1$ and $\psi =\varrho '/\varrho$. Hence $(M,g,S^*(TM))$ is a null Einstein hypersurface with non constant factor $k$. 
\end{ejemplo}

Our main interest lies in describing the case analyzed in Example \ref{ejemploGRW}. In this scenario, there are two key facts that enable us to go through our plan, namely: (i) the one-form $\tau$ vanishes along $S^*(TM)$ and (ii) the functions $\varphi\equiv 1$ and $\psi = \varrho '/\varrho$ are constant along $S^*(TM)$. Thus we can state our results in a slightly more general context, according to the following definition:

\begin{definicion}
Let $(M,g,S(TM))$ be a null screen quasi-conformal hypersurface immersed in $(\bar{M}^{n+2}_{\bar c},\bar{g})$. We say that the quasi-conformal pair $(\varphi ,\psi)$ is \emph{adapted} if 
\begin{enumerate}
\item $\tau (X)=0$, for all  $X\in\Gamma (S(TM))$,
\item $\varphi$ and $\psi$ are constant along $S(TM)$.
\end{enumerate}
\end{definicion}

We now move on into giving a local characterization of null Einstein hypersurfaces in the spirit  of the classical results of Fialkow (refer to \cite{MR1503435} and \cite{MR1075013} ). Let us consider a basis $\{ E_i\}$ of eigenvectors of $A_\xi^*$ with screen principal curvatures $\lambda_i$. Thus,  by equation (\ref{eq:Ricci}) we have for all $X,Y\in \Gamma(TM)$,
\begin{eqnarray*}
k g(X,Y)&=&\bar c ng(X,Y)+g(A_\xi^* X,Y)\text{tr}\, A_N-g(A_NX,A_\xi^*Y) \\
&=&\bar c ng(X,Y)+g(A_\xi^*X,Y)\text{tr}\, (\varphi A_\xi^*+\psi P)-g(\varphi A_\xi^*X+\psi PX,A_\xi^*Y).
\end{eqnarray*}
Therefore
\begin{equation*}
(k-\bar c n)g(X,Y)=g(A_\xi^*X,Y)(\varphi \ \text{tr}\, A_\xi^*+n\psi)-\varphi g(A_\xi^*X,A_\xi^*Y)-\psi g(PX,A_\xi^*Y).
\end{equation*}
Taking $X=Y=E_i$ in the above equation then yields
\begin{equation}\label{eq:lambda}
\varphi {\lambda_i}^2-((n-1)\psi+\varphi \ \text{tr}\, A_\xi^*)\lambda_i+(k-\bar c n)=0.
\end{equation}
Thus, the following result follows at once.

\begin{proposicion}\label{quadratic}
 Let $(M,g,S(TM))$ be a null Einstein and screen quasi-conformal hypersurface with adapted quasi-conformal pair $(\varphi ,\psi )$ immersed in a Lorentzian space form of constant curvature $(\bar{M}^{n+2}_{\bar c},\bar{g})$. Then $M$ has at most two distinct screen principal curvatures $\lambda$, $\mu$. Moreover,  if $\varphi\neq 0$ these curvatures satisfy
 \begin{equation}\label{eq:soluciones}
 \lambda +\mu =\text{tr}\, A_\xi^*+(n-1)\frac{\psi}{\varphi},\quad \lambda\mu =\frac{k-\bar c n}{\varphi}.
 \end{equation}
\end{proposicion} 
 
 \begin{proof}
Let us first notice that if $\varphi =0$ then equation (\ref{eq:lambda}) readily implies that
\begin{equation*}
\lambda_i=\frac{k-\bar{c}n}{(n-1)\psi }
\end{equation*}
Hence in this case all values $\lambda_i$ agree and we have a single screen principal curvature. 

Henceforth we assume $\varphi \neq 0$. Now, let $\lambda_1, \lambda_2$ be two distinct screen principal curvatures with $m_1,m_2$ their respective multiplicities. Then, by equation (\ref{eq:lambda}) we have 
\begin{eqnarray*}
\varphi (1-m_1)\lambda_1^2-\varphi m_2\lambda_2\lambda_1-\varphi (\text{tr}\, A_\xi^*-s)\lambda_1-(n-1)\psi \lambda_1+(k-\bar c n)&=&0,\\ 
\varphi (1-m_2)\lambda_2^2-\varphi m_1\lambda_1\lambda_2-\varphi (\text{tr}\, A_\xi^*-s)\lambda_2-(n-1)\psi \lambda_2+(k-\bar c n)&=&0,\\
\end{eqnarray*}
where $s=m_1\lambda_1+m_2\lambda_2$. After multiplying the first equation by $\lambda_2$, the second equation by $\lambda_1$ and substracting we find $(\varphi\lambda_1\lambda_2 -(k-\bar c n))(\lambda_1-\lambda_2)=0$. Since $\lambda_1\neq \lambda_2$ we have
\begin{equation}\label{eq:2lambdas}
\lambda_1\lambda_2= \frac{k-\bar c n}{\varphi}.
\end{equation}
Now assume that there are at least three distinct screen principal curvatures $\lambda_1$, $\lambda_2$, $\lambda_3$. In virtue of  equation (\ref{eq:2lambdas}) we have $\lambda_1\lambda_2=\lambda_1\lambda_3=\lambda_2\lambda_3$ and therefore at least two of the $\lambda_i$ must coincide, contradicting our assumption. Thus we have at most two different screen principal curvatures. The equations (\ref{eq:soluciones}) follow at once by direct substitution in equation (\ref{eq:lambda}).
\end{proof}

%%We now take advantage of Corollary \ref{class01} and equation (\ref{eq:Ricci}) to classify Einstein null hypersurfaces immersed in Lorentzian space forms. To this end, let us consider an Einstein null hypersurface $(M,g)$ expressed as a graph in a GRW-spacetime, as explained in Example \ref{ejemploGRW}. Thus, $(M,g,S^*(TM))$ is screen quasi-conformal and $M$ admits an induced Ricci tensor due to Proposition \ref{prop:Ricci}. 

\begin{teorema}\label{teo:clasif01}
Let $(M,g,S(TM))$ be a null Einstein and screen quasi-conformal hypersurface with adapted quasi-conformal pair $(\psi ,\varphi)$ immersed in a Lorentzian space form of constant curvature $(\bar{M}^{n+2}_{\bar c},\bar{g})$. Then $(M,g,S(TM))$ is null screen isoparametric.
\end{teorema}

\begin{proof}
In virtue of Proposition \ref{quadratic}  we have that $(M,g,S(TM))$ has at most two different screen principal curvatures. Thus, let  us consider first the case of two distinct screen principal curvatures and denote them by $\lambda ,\mu$.

%%%Moreover, notice that in the case $\psi\equiv 0$ --that is, in the screen conformal scenario-- then Lemma 6 in \cite{MR2735275} establishes the result (in fact, Note 3, p. 1187 in \cite{MR2735275}carries out the argument in detail). Hence, in what follows we will consider $\psi \not\equiv 0$.

Since $S(TM)$ is spacelike, then $T_\lambda\perp T_\mu$ and $S(TM)=T_\lambda\oplus T_\mu$ (recall Definition \ref{def:taus}). Thus $TM=Rad(M)\oplus T_\lambda\oplus T_\mu$.

We first analyze the case in which both the distributions $T_\lambda ,T_\mu$ have dimension at least 2,

%%a classical result of Ryan (see Prop. 2.3 in \cite{MR0253243}) establishes that $\lambda$ is constant along $T_\lambda$ and $\mu$ is constant along $T_\mu$. Moreover, from equations (\ref{eq:soluciones}) it follows $\mu$ is constant along $T_\lambda$ as well (notice that $\varphi \equiv 1$ and $\psi = {\sqrt{2}\rho '}/{\rho }$ are constant along $S^*(TM)$). A similar argument shows $\lambda,\mu$ are constant along $ T_\mu$ and hence $(M,g,S^*(TM))$ is null screen isoparametric. Now, if $(M,g,S^*(TM))$ has exactly two screen principal curvatures with $\dim T_\lambda =1$, $\dim T_\mu\ge 2$ then a straightforward computation in equation (\ref{eq:lambda}) gives
%%\begin{equation*}
%%\lambda=\frac{(n-2)(\bar c n-k)}{(n-1)\psi} ,\quad \mu =-\frac{(n-1)\psi}{(n-2)\varphi} ,
%%\end{equation*}
%%and thus $\lambda$ and $\mu$ are constant in $S^*(TM)$  as well. Finally, if $\dim T_\lambda=1=\dim T_\mu$  then  equations (\ref{eq:soluciones}) readily imply that $\psi \equiv 0$, contrary to our assumption.This finishes the proof in the case when there are exactly two distinct screen principal curvatures.

Notice that by equation (\ref{eq:lambda}) and Proposition \ref{quadratic} we have that 
\begin{equation*}
(A_\xi^*)^2-(\lambda +\mu) A_\xi^*+\lambda\mu=0.
\end{equation*} 
Let $T_{\lambda}^*,T_{\mu}^*:\Gamma (TM)\to \Gamma (S^*(TM))$ be defined by $T_\lambda^*=A_\xi^*-\lambda P$, $T_\mu^*=A_\xi^*-\mu P$. Notice that $T_\lambda^*(Rad(M))=\{0\}=T_\mu^*(Rad(M)) $. 
Hence, if $Y=T_\lambda^*(X)$ then we have $T_\mu^*(Y)=0$. This implies that $T_\lambda^* (S^*(TM))\subset T_\mu$, and analogously, $ T_\mu^*(S^*(TM))\subset T_\lambda$.

On the other hand, observe that for $X,Y\in T_\lambda\oplus\Gamma ( Rad(M))$, $Z\in \Gamma (TM)$ we have
\begin{eqnarray*}
X(B(Y,Z)) &=& X(\lambda )g(PY,Z)+\lambda g(\nabla_XPY,Z)\\&& +\lambda \eta (Z)B(X,Y)+\lambda g(PY,\nabla_XZ).
\end{eqnarray*}
Since $g(P(\nabla_XY),Z)=g(\nabla_XPY+\eta (Y)A_\xi^*X,Z)$ by Definition \ref{defi:derivadas}
\begin{eqnarray*}
(\nabla_XB)(Y,Z)&=&-g(T_\lambda^*(\nabla_XY),Z)+X(\lambda )g(PY,Z)\\ &&+\lambda\eta (Z)B(X,Y)+\lambda^2\eta (Y)g(PX,Z)
\end{eqnarray*}
and in virtue of item (2) in Proposition \ref{lema:lema3.3ciriaco} we have
\begin{eqnarray}\label{eq:Axilambda}
g(T_\lambda^*[X,Y],Z)&=&(X(\lambda )+\lambda\tau (X)-\lambda^2\eta (X))g(PY,Z))\\ \nonumber && -(Y(\lambda )+\lambda\tau (Y)-\lambda^2\eta (Y))g(PX,Z)).
\end{eqnarray}
As a consequence of  $T_\lambda\perp T_\mu$ we have that  $g(T_\lambda^*[X,Y],Z)=0$ 
 and thus $T_\lambda^*[X,Y]\perp T_\mu$. However, $T_\lambda^*[X,Y]\in T_\mu$. Hence $T_\lambda^*[X,Y]=0$ and $[X,Y]\in T_\lambda$. A similar argument shows that $[X,Y]\in T_\mu$ for $X,Y\in T_\mu$. Finally, from equation (\ref{eq:Axilambda}) we have
\begin{equation*}
(X(\lambda )+\lambda\tau (X)-\lambda^2\eta (X))g(PY,Z))=(Y(\lambda )+\lambda\tau (Y)-\lambda^2\eta (Y))g(PX,Z))
\end{equation*} 
for all $Z\in \Gamma (TM)$. It follows that 
\begin{equation*}
(X(\lambda )+\lambda\tau (X)-\lambda^2\eta (X))PY=(Y(\lambda )+\lambda\tau (Y)-\lambda^2\eta (Y))PX.
\end{equation*}
If for some $X\in T_\lambda$ we have that $X(\lambda )+\lambda\tau (X)-\lambda^2\eta (X)\neq 0$ then by the last equation we find that $\dim T_\lambda =1$, which contradicts our assumption.
Hence, if we assume that $\dim T_\lambda\ge 2$ and $\dim T_\mu\ge 2$,  we have
\begin{equation*}
X(\lambda )+\lambda\tau (X)-\lambda^2\eta (X)=0 
\end{equation*}
for all $X\in T_\lambda\oplus\Gamma (Rad(M))$. Since $\tau$ vanishes in $S(TM)$, the above equation implies that $X(\lambda )=0$ in $T_\lambda$. Moreover, from equations (\ref{eq:soluciones}) it follows $\mu$ is constant along $T_\lambda$ as well.  A similar argument shows $\lambda,\mu$ are constant along $ T_\mu$ and hence $(M,g,S^*(TM))$ is null screen isoparametric.

Now, if $(M,g,S(TM))$ has exactly two screen principal curvatures with $\dim T_\lambda =1$, $\dim T_\mu\ge 2$ then a straightforward computation in equation (\ref{eq:lambda}) gives
\begin{equation*}
\lambda=\frac{(n-2)(\bar c n-k)}{(n-1)\psi} ,\quad \mu =-\frac{(n-1)\psi}{(n-2)\varphi} ,
\end{equation*}
and thus $\lambda$ and $\mu$ are constant in $S(TM)$  as well. Finally, if $\dim T_\lambda=1=\dim T_\mu$  then  equations (\ref{eq:soluciones}) readily imply that $\psi \equiv 0$, contrary to our assumption made in Remark \ref{rem:psinocero}.This finishes the proof in the case when there are exactly two distinct screen principal curvatures.

%%On the other hand, if $\dim T_\lambda =1$ and $n\ge 3$, then $\dim T_\mu>1$.  Hence classical result of Ryan (see Prop. 2.3 in \cite{MR0253243}) coupled with the equations (\ref{eq:soluciones}) establish that both $\lambda$ and $\mu$ are constant, thus yielding the desired result in this case also.  

On the other hand, let us assume $(M,g,S(TM))$ has only one screen principal curvature, and denote it by $\lambda$. 
From equation (\ref{eq:lambda}) we have 
\begin{equation*}
\varphi\lambda^2-((n-1)\psi +n\varphi\lambda )\lambda+(k-\bar{c}n)=0
\end{equation*}
hence the condition that ensures that the above quadratic equation has exactly one solution reads
\begin{equation}\label{eq:unasolalambda}
(n-1)\psi^2+4\varphi (k-\bar{c}n)=0.
\end{equation}
and hence
\begin{equation}\label{eq:lalambda}
\lambda =-\frac{\psi}{2\varphi }.
\end{equation}
From the above formula we can readily see that $\lambda$ is constant along $S(TM)$ thus concluding the proof.
\end{proof}

We are now ready to state the main local classification results for null Einstein hypersurfaces of $(\bar{M}^{n+2}_{\bar c},\bar{g})$.

\begin{teorema}
Let $(M,g,S(TM))$ be a null Einstein and screen quasi-conformal hypersurface with adapted quasi-conformal pair $(\varphi ,\psi )$ immersed in a Lorentzian space form of constant curvature $(\bar{M}^{n+2}_{\bar c},\bar{c})$.  If $(M,g,S(TM))$ has exactly two distinct screen principal curvatures $\lambda$, $\mu$, then $(M,g,S(TM))$ is locally diffeomorphic to a product $M=\ell\times M_\lambda\times M_{\mu}$ where $\ell$ is a null geodesic and $M_\lambda, M_\mu$ are two Riemannian manifolds. Moreover, 
 if $\bar{c}=0$ then one of the screen principal curvatures vanishes.
\end{teorema}

\begin{proof}
From Theorem \ref{teo:umbilica}, $S(TM)$ is integrable, so let us denote by $M'$ one of its integral manifolds. Let us notice that in view of Lemma \ref{lema:distribuciones} we have that both $T_\lambda$ and $T_\mu$ are parallel distributions on $S(TM)$. Thus, by de Rham's decomposition theorem we have that locally $M'$ is isometric to $M_\lambda\times M_\mu$ where $M_\lambda$, $M_\mu$ are integral manifolds of $T_\lambda$, $T_\mu$.  Moreover, since the integral curves of $Rad (M)$ are pregeodesics with respect to $\bar\nabla$ we have that $Rad(M)$ is involutive and hence $\bar{M}$ is locally isometric to $\ell\times M_\lambda\times M_{\mu}$, where $\ell$ is a null curve. Further, by Theorem \ref{teo:clasif01} and Corollary \ref{class01}, we have that in the case $\bar{c}=0$ one of the null principal curvatures vanishes. \end{proof}

\begin{teorema}\label{teo:class03}
Let $(M,g,S(TM))$ be a null Einstein and screen quasi-conformal hypersurface with adapted quasi-conformal pair $(\varphi ,\psi )$ immersed in a Lorentzian space form of constant curvature $(\bar{M}^{n+2}_{\bar c},\bar{g})$. If $(M,g,S(TM))$ has exactly one screen principal curvature $\lambda$, then $(M,g,S(TM))$ is locally diffeomorphic to a product $M=\ell\times M'$ where $\ell$ is a null geodesic and $M'$ is a Riemannian manifold of constant curvature. Moreover, $M'$ is a totally umbilical submanifold of $(\bar{M}^{n+2}_{\bar c},\bar{g})$ of codimension 2.
\end{teorema}

%%%%\textcolor{blue}{Estimados: para que aparezca el do Carmo/Dajczer en las referencias hay que citarlo en alg\'un lugar. Este es el MR: \cite{MR694383}. La ref. 11 del do Carmo tiene este MR: \cite{MR0253243}}

\begin{proof}
Let $\lambda$ be the only screen principal direction. Notice that in this case $M$ is totally umbilical. Hence, by Theorem \ref{teo:umbilica}  we have that $M'=M_\lambda$ is totally umbilical as a codimension 2 submanifold of $\bar{M}_{\bar c}^{n+2}$.

On the other hand, by equations (\ref{eq:2.21Ciriaco}) and (\ref{eq:2.22Ciriaco1}) we have that for all $X,Y,Z\in T_\lambda $, $W\in\Gamma (TM')$
\begin{eqnarray*}
g(\bar{R}(X,Y)Z,W) &=& g(R(X,Y),Z,W)\\
&=& g (R^*(X,Y)Z,W) +\lambda (\varphi\lambda +\psi)g(X,Z)g(Y,W)\\ &&-\lambda (\varphi\lambda +\psi)g(Y,Z)g(X,W).
\end{eqnarray*}
Thus
\begin{equation*}
R^*(X,Y)Z=[\lambda (\varphi\lambda+\psi )+\bar{c}](g(Y,Z)X-g(X,Z)Y).
\end{equation*}
Moreover, let us notice that equations (\ref{eq:unasolalambda}) and (\ref{eq:lalambda}) imply
\begin{equation*}
\lambda (\lambda\varphi+\psi)+\bar{c}=-\frac{\psi^2}{4\varphi}+\bar{c}=\frac{k-\bar{c}}{n-1}
\end{equation*}
and therefore $M'$ is a Riemannian manifold of constant curvature ${\displaystyle \frac{k-\bar{c}}{n-1}}$.
\end{proof}

\begin{corolario}
There are no Ricci flat null hypersurface $(M,g,S^*(TM))$ in de Sitter space $\mathbb{S}^{n+2}_1=(\bar{M}^{n+2}_1,\bar{g})$ having exactly one screen principal curvature.
\end{corolario}

\begin{proof}
By Theorem \ref{teo:class03} we have that in this case $M'\subset S_t$ is a Riemannian manifold of constant curvature ${\displaystyle -\frac{{1}}{n-1}}<0$,  which is absurd since $S_t$ is a round sphere.
\end{proof}

\begin{corolario}
Let $(M,g,S^*(TM))$ be a null Einstein hypersurface with $\bar c=1,-1$, $n>2$. If $\psi$ changes sign on $M$ and
\begin{equation*}
{{k>\bar cn}>0},
\end{equation*}
then $(M,g,S^*(TM))$ is locally diffeomorphic to a product $M=\ell\times M'$ where $\ell$ is a null geodesic and $M'$ is a Riemannian manifold of constant curvature.
\end{corolario}

\begin{proof}
Let us assume $(M,g,S^*(TM))$ has exactly two screen principal curvatures $\lambda$, $\mu$. According to equations (\ref{eq:soluciones}) the condition $k>\bar cn$ implies that $\lambda$ and $\mu$ have the same sign (recall that $\varphi\equiv 1$ in $S^*(TM)$), and in particular neither $\lambda$ nor $\mu$ can be zero. Moreover, we have that
\begin{equation}\label{eq:traza}
(p-1)\lambda+(n-p-1)\mu =\frac{\psi}{\varphi}
\end{equation}
for some $1<p<n-1$ and hence the left hand side of equation (\ref{eq:traza}) has a definite sign. However, the right hand side of equation (\ref{eq:traza}) changes sign by hypothesis. Thus we can conclude that there is only one screen principal curvature and the results follows from Theorem \ref{teo:class03}.
\end{proof}

\section*{Acknowledgements}

Didier A. Solis would like to thank the Centro de Investigaci\'on en Matem\'aticas (CIMAT) unidad M\'erida for the kind hospitality during a sabbatical leave in which part of this research was developed.

\section*{References}

\bibliographystyle{plain}
\bibliography{references5}

\end{document}